\documentclass[12pt]{amsart}
 
\pdfoutput=1
\usepackage[margin=1in]{geometry}
\usepackage{amsmath,amsthm,amssymb,amsrefs,bbm,color,esint,esvect,float,graphicx,mathrsfs}
\usepackage[bookmarksnumbered, colorlinks, plainpages,linkcolor=blue,anchorcolor=blue,citecolor=blue,urlcolor=blue]{hyperref}
\usepackage{amsmath,amssymb,amsthm}

\usepackage{epstopdf}
\AppendGraphicsExtensions{.tif}

\usepackage{enumitem}
\setlist[enumerate]{itemsep=3pt,topsep=4pt}
\setlist[itemize]{itemsep=3pt,topsep=4pt}

 \newtheorem{thm}{Theorem}[section]
 \newtheorem{lemma}{Lemma}[section]

\newtheorem{proposition}{Proposition}[section]

\numberwithin{equation}{section}

\begin{document}

\title[A dimension-free weak-type $(1,1)$ bound for the vector Riesz transform on $\mathbb{R}^n$]{A dimension-free weak-type $(1,1)$ bound for the vector Riesz transform on $\mathbb{R}^n$}

\author{Yuyuan Ouyang}
\address{Yuyuan Ouyang, School of Mathematical and Statistical Sciences, Clemson University, Clemson, SC 29634, USA}
\email{yuyuano@clemson.edu}

\author{Daniel Spector}
\address{Daniel Spector, Department of Mathematics, National Taiwan Normal University, No. 88, Section 4, Tingzhou Road, Wenshan District, Taipei City, Taiwan 116, R.O.C.
\newline
Department of Mathematics, University of Pittsburgh, Pittsburgh, PA 15261 USA
}
\email{spectda@gapps.ntnu.edu.tw}
\thanks{D. Spector is supported by the National Science and Technology Council of Taiwan under research grant number 113-2115-M-003-017-MY3.}

\author{Cody B. Stockdale}
\address{Cody B. Stockdale, School of Mathematical and Statistical Sciences, Clemson University, Clemson, SC 29634, USA}
\email{cbstock@clemson.edu}
\thanks{C. B. Stockdale is supported by the National Science Foundation, DMS grant No. 2555710}

\maketitle

\begin{abstract}
We show that the best constant in the weak-type $(1,1)$ bound for the vector Riesz transform on $\mathbb{R}^n$ is at most $2$, independent of the dimension $n$. The proof relies on a new decomposition of the input data involving an obstacle problem for the fractional Laplacian and an associated Lewy-Stampacchia type estimate on an unbounded domain. This settles a problem posed by E. M. Stein at the 1986 International Congress of Mathematicians.\looseness=-1
\end{abstract}


\section{Introduction}
Let $\mathbb{R}^n$ denote the Euclidean space of $n\in\mathbb{N}$ dimensions. For $j \in \{1,2,\ldots,n\}$,  define the $j^{\text{th}}$ Riesz transform of a function $f \in C^\infty_c(\mathbb{R}^n)$ by 
\begin{align*}
	R_jf(x):=\frac{\Gamma\left(\frac{n+1}{2}\right)}{\pi^{\frac{n+1}{2}}} p.v. \int_{\mathbb{R}^n}\frac{y_j-x_j}{|y-x|^{n+1}}f(y)\,dy.
\end{align*}
An outstanding open question of E. M. Stein \cite[Problem b on p.~203]{Stein86} asks whether the classical weak-type $(1,1)$ inequality
\begin{align}
    \|R_jf\|_{L^{1,\infty}(\mathbb{R}^n)} := \sup_{\lambda>0} \lambda |\{x \in \mathbb{R}^n\colon |R_j f(x)|>\lambda \}| \leq C_n \|f\|_{L^1(\mathbb{R}^n)}\label{wl1}
\end{align}
for all $f \in L^1(\mathbb{R}^n)$ can be established with a constant $C_n>0$ independent of the dimension $n$. \looseness=-1

The main result of this article is an affirmative answer to Stein's question. In fact, we prove the following stronger assertion for the full vector Riesz transform $Rf := (R_1f,R_2f,\ldots,R_nf)$. 
\begin{thm}\label{MainTheorem}
If $f \in L^1(\mathbb{R}^n)$, then 
\begin{align*}
	\|Rf\|_{L^{1,\infty}(\mathbb{R}^n)} 
\leq 2\|f\|_{L^1(\mathbb{R}^n)}.
\end{align*}
\end{thm}
\noindent

Before this paper, the best-known dimensional constant in \eqref{wl1} was due to Janakiraman \cite{J2004}, who showed that the bound \eqref{wl1} holds with $C_n=c \log(n)$ for some absolute constant $c>0$. We remark that this $\log(n)$ dependence for the component Riesz transforms only implies the corresponding weak-type bound for the vector Riesz transform with an asymptotic dependence of $\sqrt{n}\log(n)$. On the other hand, Theorem \ref{MainTheorem} gives that the weak-type $(1,1)$ bound for the vector Riesz transform $R$ holds with constant $2$ for any dimension, and therefore \eqref{wl1} is also valid with the same constant. It is also interesting to observe that the constant $2$ appearing in Theorem \ref{MainTheorem} matches the constant in the dimension-free strong-type bounds for $R$ on $L^p(\mathbb{R}^n)$ obtained by Ba\~{n}uelos and Wang \cite[Corollary 4.5]{BW}; see also Dragi\v{c}evi\'c and Volberg \cite[Corollary 0.2]{DV06}. 
We refer to \cite{BO,SS} for further historical remarks on dimensionally-dependent estimates for the Riesz transforms, and to \cite{A2011} for a parallel historical discussion concerning maximal functions.
\looseness=-1

Prior progress towards the Riesz transform dimensional weak-type $(1,1)$ problem relied on adaptations of Calder\'on-Zygmund decompositions which, although are valid for a broader class of singular integrals, did not take into account the particular structure of the Riesz transform. Following these methods, dimensional independence is lost when treating terms involving the ``bad" pieces, as one must integrate the kernel on annular regions surrounding their supports.  The following new decomposition, which is adapted to the Riesz transform's structure as an operator satisfying $R= \nabla(-\Delta)^{-1/2}$, is the crux of the proof of Theorem \ref{MainTheorem}.  Although our application only requires the case when the smoothness parameter $\alpha$ equals $1$, as it may be of interest, we establish the decomposition for general $0<\alpha<2$.  In what follows, $H^{\alpha}(\mathbb{R}^n)$ denotes the Sobolev space consisting of $u \in L^2(\mathbb{R}^n)$ such that $(-\Delta)^{\alpha/2}u \in L^2(\mathbb{R}^n)$.  \looseness=-1

\begin{thm}\label{thm:decomp}
If $f\in L^1(\mathbb{R}^n)\cap L^2(\mathbb{R}^n)$ is nonnegative, $\lambda>0$, and $0<\alpha<2$, then there exist nonnegative functions $\mu \in L^1(\mathbb{R}^n)\cap L^\infty(\mathbb{R}^n)$ and $u \in H^\alpha(\mathbb{R}^n)\cap L^1(\mathbb{R}^n)$ such that 
$$
	f = \mu + (-\Delta)^{\alpha/2}u,
$$
where
\begin{enumerate}
\addtolength{\itemsep}{2pt}
\item $\|\mu\|_{L^{\infty}(\mathbb{R}^n)} \leq \lambda$ and $\|\mu\|_{L^1(\mathbb{R}^n)} = \|f\|_{L^1(\mathbb{R}^n)}$,
\item $\mu=\lambda$ almost everywhere on $\Omega:=\{x\in\mathbb{R}^n\colon u(x)>0\}$,  and
\item $|\Omega|\leq \frac{1}{\lambda} \|f\|_{L^1(\mathbb{R}^n)}$.
\end{enumerate}
Moreover, if $\alpha\ge1$, then $u\in H^1(\mathbb{R}^n)$ and $\nabla u=0$ almost everywhere on $\mathbb{R}^n\setminus\Omega$. 
\end{thm}

Theorem \ref{thm:decomp} is referred to in certain literature as partial balayage \cite{GS94,GR18} and in other literature as a divisible sandpile and odometer \cite{LP09,LP10}, though we find it useful to adopt the language of variational inequalities and obstacle problems \cite{KS80,SV13}.
The case $\alpha=2$ is classical; see \cite{KS80} for an extensive development.  
In the fractional case, a bounded domain analogue of our Theorem \ref{thm:decomp} has been established in the work of Servadei and Valdinoci \cite{SV13}.  As the reader can verify, in addition to the minor technical differences we implement to handle the unbounded case, our emphasis differs from the presentation in \cite{SV13}.  This shift in perspective, along with the many small technical modifications, motivates the complete presentation of details we give in the sequel.

We now prove Theorem \ref{MainTheorem}, assuming Theorem \ref{thm:decomp}.
\begin{proof}[Proof of Theorem \ref{MainTheorem}]
By density, we may assume that $f \in L^1(\mathbb{R}^n)\cap L^2(\mathbb{R}^n)$.
Write $f=f^+-f^-$ and let $\lambda>0$.   An application of Theorem~\ref{thm:decomp} with $\alpha=1$ to $f^+$ and $f^-$
at level $\lambda$ yields nonnegative functions $\mu_\pm$ and  $u_\pm$ such that 
$$
	f^\pm=\mu_\pm+(-\Delta)^{1/2} u_\pm,
$$
$\|\mu_\pm\|_{L^{\infty}(\mathbb{R}^n)} \leq \lambda$, $\|\mu_\pm\|_{L^1(\mathbb{R}^n)}=\|f^\pm\|_{L^1(\mathbb{R}^n)}$, $\mu_\pm=\lambda$ almost everywhere on $\Omega_\pm:=\{u_\pm>0\}$,
$\lambda|\Omega_\pm|\le \|f^\pm\|_{L^1(\mathbb{R}^n)}$, and $\nabla u_\pm=0$ almost everywhere on $\mathbb{R}^n\setminus\Omega_\pm$. Define $\mu:=\mu_+-\mu_-$, $u:=u_+-u_-$, and $\Omega:=\Omega_+\cup\Omega_-$, so that
\begin{align*}
f&=\mu+(-\Delta)^{1/2} u.
\end{align*}
As this is an identity of $L^2(\mathbb{R}^n)$ functions, this implies
\begin{align*}
Rf&=R\mu+\nabla u.
\end{align*}
Both $\nabla u_+$ and $\nabla u_-$ vanish almost everywhere away from $\Omega$, so
$\nabla u=0$ almost everywhere on $\mathbb{R}^n\setminus\Omega$ and therefore, up to a null set, we have
\[
\{|R f|>\lambda\}\subseteq\Omega\cup\{|R\mu|>\lambda\}.\]
By subadditivity, 
\[
\lambda|\Omega|\;\le\;\lambda|\Omega_+|+\lambda|\Omega_-|\;\le\; \|f^+\|_{L^1(\mathbb{R}^n)} + \|f^-\|_{L^1(\mathbb{R}^n)}\;=\;\|f\|_{L^1(\mathbb{R}^n)}.
\]
Meanwhile, as $0\le\mu_\pm\le\lambda$ we have $|\mu|\le\lambda$ pointwise, and so
\[
\|\mu\|_{L^2(\mathbb{R}^n)}^2\leq \lambda\int_{\mathbb{R}^n}|\mu|\,dx
=\lambda\int_{\mathbb{R}^n}(\mu_++\mu_-)\,dx=\lambda\,(\|f^+\|_{L^1(\mathbb{R}^n)}+\|f^-\|_{L^1(\mathbb{R}^n)})=\lambda \|f\|_{L^1(\mathbb{R}^n)}.
\]
By Chebyshev's inequality and the boundedness of the Riesz transform on $L^2(\mathbb{R}^n)$, we have 
\[
\lambda|\{|R\mu|>\lambda\}|\le\lambda^{-1}\|R\mu\|_{L^2(\mathbb{R}^n)}^2=\lambda^{-1}\|\mu\|_{L^2(\mathbb{R}^n)}^2\le \|f\|_{L^1(\mathbb{R}^n)}.\]
The result follows upon taking a supremum over all $\lambda>0$. 
\end{proof}

The body of this paper is focused on establishing Theorem \ref{thm:decomp}.  We give a brief outline here, focusing on the case $\alpha=1$ for the ease of presentation.  For fixed $f\in L^1(\mathbb{R}^n)\cap L^2(\mathbb{R}^n)$, $\lambda>0$, and $v \in H^{1/2}(\mathbb{R}^n)\cap L^1(\mathbb{R}^n)$, we define the energy functional
\begin{align*}
J_\lambda(v)&:=\frac12 \frac{\Gamma\!\big(\tfrac{n+1}2\big)}{\pi^{(n+1)/2}} \iint_{\mathbb{R}^{2n}}\frac{|v(x)-v(y)|^2}{|x-y|^{n+1}}\,dxdy -\int_{\mathbb{R}^n}(f-\lambda)v\,dx. 
\end{align*}
Following the direct method of Tonelli, a combination of compactness and a lower-semicontinuity argument yields that $J_\lambda$ admits a unique minimizer over the cone of nonnegative functions
\[
\mathcal{K}:=\{v\in H^{1/2}(\mathbb{R}^n)\cap L^1(\mathbb{R}^n)\colon  v\ge0\ \text{almost everywhere on }\mathbb{R}^n\}.
\]
Letting $u$ denote this minimizer and writing  
$$
\mathcal{E}(u,v):= \frac12 \frac{\Gamma\!\big(\tfrac{n+1}2\big)}{\pi^{(n+1)/2}} \iint_{\mathbb{R}^{2n}}\frac{(u(x)-u(y))(v(x)-v(y))}{|x-y|^{n+1}}\,dxdy
$$
for a suitably normalized inner product on $H^{1/2}(\mathbb{R}^n)$, a standard approach in the theory of variational inequalities yields that $u$ satisfies
$$
\mathcal{E}(u,v-u) \ge  \int_{\mathbb{R}^n}(f-\lambda)(v-u)\,dx
$$
for all $v\in\mathcal K$, along with the energy identity
$$
\mathcal E(u,u)=\int_{\mathbb{R}^n}(f-\lambda)u\,dx. 
$$
This implies that the distribution
$$
\eta := (-\Delta)^{1/2}u - (f-\lambda)
$$
is a positive Radon measure.  With the choice of a suitable test function, one finds that this measure is absolutely continuous with respect to the Lebesgue measure, and moreover, enjoys the Lewy--Stampacchia type estimate
$$
0\le\eta\le(\lambda-f)^+\le\lambda
$$
almost everywhere.  Defining $\mu:=\lambda-\eta$, one obtains the claimed decomposition
$$
f = \mu + (-\Delta)^{1/2}u.
$$

The paper is organized as follows. In Section \ref{preliminaries}, we recall some relevant results from the theory of Sobolev spaces.  In Section \ref{fractional}, we recall several results concerning the Gagliardo semi-norms and fractional Sobolev spaces.  In Section \ref{existence_Veq}, we introduce the energies of interest to us in this paper.  For these energies, we establish the existence and uniqueness of minimizers over the suitable positive cones and derive the variational inequalities and energy identities satisfied by the minimizers.  In Section \ref{sec:cap}, we derive the Lewy--Stampacchia type estimates and establish further properties of the functions $\mu$ and $u$.  Finally, in Section \ref{decomposition} we put together the various pieces developed in the preceding sections to prove Theorem \ref{thm:decomp}.  


\section{Preliminaries}\label{preliminaries}

We record several results that will be useful in the sequel.  We first prove the following lemma, which asserts the validity of the chain rule in $W^{1,1}_{\text{loc}}(\mathbb{R}^n)$ when one composes with certain $C^1(\mathbb{R})$ functions. Here, $W^{1,1}_{\text{loc}}(\mathbb{R}^n)$ denotes the locally-integrable functions $w$ such that $\nabla w \in L^1_{\text{loc}}(\mathbb{R}^n)$.
\begin{lemma}\label{lem:chain}
Let $\Phi\in C^1(\mathbb R)$ with $\Phi(0)=0$ and $\Phi' \in L^{\infty}(\mathbb{R})$. If $w\in
W^{1,1}_{\mathrm{loc}}(\mathbb{R}^n)$, then $\Phi\circ w\in W^{1,1}_{\mathrm{loc}}(\mathbb{R}^n)$ and
$\nabla(\Phi\circ w)=\Phi'(w)\nabla w$ almost everywhere. 
\end{lemma}

\begin{proof}
Fix a ball $B\subseteq \mathbb{R}^n$ and denote by $\rho_{1/k}$ a sequence of standard smooth mollifiers. Define $w_k:=w\ast\rho_{1/k}$ so that $w_k\to w$ and $\nabla w_k\to\nabla w$ in $L^1(B)$. After passing to a subsequence, we may assume that $w_k\to w$ almost everywhere on $B$. Since $\Phi \in C^1(\mathbb{R})$ and $w_k$ is smooth, the classical chain rule gives $\nabla\Phi(w_k)=\Phi'(w_k)\nabla w_k$ on $B$. Now $|\Phi(w_k)-\Phi(w)|\le\|\Phi'\|_{L^{\infty}(\mathbb{R})} |w_k-w|\to0$ in $L^1(B)$, and 
\[
\|\Phi'(w_k)\nabla w_k-\Phi'(w)\nabla w\|_{L^1(B)}
\le\|\Phi'\|_{L^{\infty}(\mathbb{R})}\|\nabla w_k-\nabla w\|_{L^1(B)}+\int_B|\Phi'(w_k)-\Phi'(w)||\nabla w| \,dx,
\]
where the last term tends to $0$ by dominated convergence as $\Phi'$ is continuous and bounded, $w_k\to w$ almost everywhere, and $\nabla w\in L^1(B)$. Passing to the limit in the definition of the weak derivative gives the claim on $B$.  The general result follows as $B$ was arbitrary.
\end{proof}

The next lemma is due to Stampacchia and plays an important role in our analysis.
\begin{lemma}\label{lem:stamp}
If $w\in W^{1,1}_{\mathrm{loc}}(\mathbb{R}^n)$ and $c\in\mathbb R$, then $(w-c)^\pm\in
W^{1,1}_{\mathrm{loc}}(\mathbb{R}^n)$ with
\[
\nabla (w-c)^+=\mathbf{1}_{\{w>c\}}\nabla w \qquad \text{and}\qquad
\nabla (w-c)^-=-\mathbf{1}_{\{w<c\}}\nabla w.
\]
Consequently, $\nabla w=0$ almost everywhere on $\{x\in\mathbb{R}^n\colon w(x)=c\}$.
\end{lemma}

\begin{proof}
Set $\Phi_\delta(s):=\sqrt{(s-c)^2+\delta^2}-\sqrt{c^2+\delta^2}$ for $\delta>0$. Then
$\Phi_\delta\in C^1(\mathbb{R})$, $\Phi_\delta(0)=0$, and $\Phi_\delta'\in L^{\infty}(\mathbb{R})$, so Lemma~\ref{lem:chain} gives
\[
\nabla \Phi_\delta(w)=\frac{w-c}{\sqrt{(w-c)^2+\delta^2}}\,\nabla w .
\]
We have $\Phi_\delta(w)\to|w-c|-|c|$ pointwise as $\delta\rightarrow 0$ with
$|\Phi_\delta(w)|\le|w-c|+|c|$, hence also in $L^1_{\mathrm{loc}}(\mathbb{R}^n)$ by dominated convergence. Further, 
\[
\frac{w-c}{\sqrt{(w-c)^2+\delta^2}}\,\nabla w\;\longrightarrow\;\operatorname{sgn}(w-c)\,\nabla w
\]
in $L^1_{\mathrm{loc}}(\mathbb{R}^n)$, again by dominated convergence, where $\operatorname{sgn}(0):=0$. 
Passing to the limit in the definition of the weak derivative, we have that $|w-c|\in W^{1,1}_{\mathrm{loc}}$ with $\nabla|w-c|=\operatorname{sgn}(w-c)\nabla w$.

Since $(w-c)^+=\tfrac12\big(|w-c|+(w-c)\big)$, we obtain $\nabla(w-c)^+=\tfrac12\big(\operatorname{sgn}(w-c)+1\big)\nabla w=\mathbf{1}_{\{w>c\}}\nabla w$, and similarly for $(w-c)^-$. Finally, $w-c=(w-c)^+-(w-c)^-$ gives $\nabla w=\mathbf{1}_{\{w>c\}}\nabla w+\mathbf{1}_{\{w<c\}}\nabla w=\mathbf{1}_{\{w\ne c\}}\nabla w$, i.e., $\mathbf{1}_{\{w=c\}}\nabla w=0$ almost everywhere.
\end{proof}


\section{Fractional Sobolev Space Preliminaries}\label{fractional}
We use the convention
$$
	\widehat f(\xi)=\int_{\mathbb{R}^n}f(x)e^{-2\pi i x\cdot\xi}\,dx,
$$
for the Fourier transform, and, for $\alpha>0$, write $(-\Delta)^{\alpha/2}$ to denote the fractional Laplacian of order $\alpha/2$, i.e., the Fourier multiplier with symbol $(2\pi|\xi|)^{\alpha}$. We further define the quadratic form or inner product associated with the fractional Laplacian by
$$
	\mathcal E_\alpha(v,w)=\int_{\mathbb{R}^n}(2\pi|\xi|)^\alpha\widehat v(\xi)\,\overline{\widehat w(\xi)}\,d\xi.
$$
We use the following shorthand notation for the square of the Sobolev semi-norm associated with the inner product $\mathcal E_\alpha$: 
$$
\mathcal E_\alpha(v):=\mathcal E_\alpha(v,v).
$$
This semi-norm admits an associated energy space that will be of interest to us
$$
X_\alpha:=\{v\in L^1(\mathbb{R}^n)\colon  \mathcal E_\alpha(v)<\infty\}.
$$
We will use the following Gagliardo representation of $\mathcal E_\alpha$, see, e.g. \cite[Section 3]{DNPV12}.
\begin{lemma}\label{lem:gag}
If $0<\alpha<2$, then 
\[
\mathcal E_\alpha(v,w)=\frac{C(n,\alpha)}{2}\iint_{\mathbb{R}^{2n}} \frac{\big(v(x)-v(y)\big)\big(w(x)-w(y)\big)}{|x-y|^{\,n+\alpha}}\,dxdy
\]
 for all $v,w\in X_\alpha$, where $$C(n,\alpha):=\frac{\alpha\,2^{\alpha-1}\,\Gamma\!\big(\tfrac{n+\alpha}2\big)}{\pi^{n/2}\,\Gamma\!\big(1-\tfrac\alpha2\big)}$$ and the double integral converges absolutely. Equivalently, 
$$
(-\Delta)^{\alpha/2}v(x)=C(n,\alpha)\,p.v.\int_{\mathbb{R}^n}\frac{v(x)-v(y)}{|x-y|^{n+\alpha}}\,dy
$$
for all $v \in X_{\alpha}$. 
\end{lemma}

We next record an important property on the sign of the inner product of nonnegative functions with disjoint supports.
\begin{lemma}\label{lem:markov}
If $0<\alpha<2$ and $v,w\in X_\alpha$ are nonnegative functions with $vw=0$ almost everywhere, then $\mathcal E_\alpha(v,w)\le0$.
\end{lemma}

\begin{proof}
By the hypotheses,  we have that $(v(x)-v(y))(w(x)-w(y))=-v(x)w(y)-v(y)w(x)\le0$. Therefore, Lemma~\ref{lem:gag} gives 
$$
	\mathcal E_\alpha(v,w)=-\frac{C(n,\alpha)}2\iint_{\mathbb{R}^{2n}} \frac{v(x)w(y)+v(y)w(x)}{|x-y|^{n+\alpha}}\,dxdy \leq 0,
$$
as desired.
\end{proof}

The following lemma asserts that the space $X_\alpha$ is closed under composition with $1$-Lipschitz functions and that such compositions decrease the energy $\mathcal E_\alpha$. As a consequence, we have that $v^\pm\in X_\alpha$ with $\mathcal{E}_{\alpha}(v^{\pm}) \leq \mathcal{E}_{\alpha}(v)$ whenever $v \in X_{\alpha}$. 
\begin{lemma}\label{lem:contr}
If $\Phi\colon \mathbb R\to\mathbb R$ is a $1$-Lipschitz function which satisfies $\Phi(0)=0$, $0<\alpha<2$, and $v\in X_\alpha$, then $\Phi\circ v\in X_\alpha$ with $\mathcal E_\alpha(\Phi\circ v)\le\mathcal E_\alpha(v)$. 
\end{lemma}

\begin{proof}
We have that $\Phi \circ v\in L^1(\mathbb{R}^n)$ since $\Phi$ is Lipschitz,  $\Phi(0)=0$, and $v \in L^1(\mathbb{R}^n)$. 
Moreover, the energy bound is immediate from Lemma~\ref{lem:gag} and $|\Phi(v(x))-\Phi(v(y))|\le|v(x)-v(y)|$. 
\end{proof}

An important consequence of the Markov property we will make use of is the following bound for the inner product of a function and it's negative part.  
\begin{lemma}\label{lem:negpart}
If $0<\alpha<2$ and $v\in X_\alpha$, then $\ \mathcal E_\alpha(v,v^-)\le-\mathcal E_\alpha(v^-)\le0$ and
$\mathcal E_\alpha(v,v^+)\ge\mathcal E_\alpha(v^+)\ge0$.
\end{lemma}

\begin{proof}
Writing $v=v^+-v^-$, we have that $\mathcal E_\alpha(z,z^-)=\mathcal E_\alpha(z^+,z^-)-\mathcal E_\alpha(z^-)$ and $\mathcal E_\alpha(z,z^+)=\mathcal E_\alpha(z^+)-\mathcal E_\alpha(z^-,z^+)$. The result follows from
Lemma~\ref{lem:markov} as $v^\pm\ge0$ and $v^+v^-=0$.
\end{proof}

The following Gagliardo-Nirenberg type interpolation inequality shows the space $X_\alpha$ embeds continuously into $L^2(\mathbb{R}^n)$. Below, $\omega_n$ denotes the volume of the unit ball in $\mathbb{R}^n$. \looseness=-1
\begin{lemma}\label{lem:interp}
If $0< \alpha< 2$ and $v\in X_\alpha$, then
\[
\|v\|_{L^2(\mathbb{R}^n)}^2\;\le\;K_{n,\alpha}\,\|v\|_{L^1(\mathbb{R}^n)}^{\frac{2\alpha}{n+\alpha}}\, \mathcal E_\alpha(v)^{\frac{n}{n+\alpha}},
\]
where $K_{n,\alpha}:=\frac{n+\alpha}{n}\big(\frac n\alpha\big)^{\frac{\alpha}{n+\alpha}} \big(\frac{\omega_n}{(2\pi)^{n}}\big)^{\frac{\alpha}{n+\alpha}}$.  
\end{lemma}

\begin{proof}
For $r>0$, Plancherel gives
\[
\|v\|_{L^2(\mathbb{R}^n)}^2=\int_{|\xi|<r}|\widehat v(\xi)|^2\,d\xi+\int_{|\xi|\ge r}|\widehat v(\xi)|^2\,d\xi \le \omega_nr^n\|v\|_{L^1(\mathbb{R}^n)}^2+(2\pi r)^{-\alpha}\mathcal{E}_{\alpha}(v).
\]
The right-hand side is minimized when $r^{n+\alpha} = \frac{\alpha\mathcal{E}_{\alpha}(v)}{(2\pi)^{\alpha}n\omega_n\|v\|_{L^1(\mathbb{R}^n)}^2}$, and the minimum equals
$K_{n,\alpha}\,\|v\|_{L^1(\mathbb{R}^n)}^{\frac{2\alpha}{n+\alpha}}\, \mathcal E_\alpha(v)^{\frac{n}{n+\alpha}}$, as desired. 
\end{proof}

\section{Energies, Minimizers, and Variational Inequalities} \label{existence_Veq}
Given $f \in L^1(\mathbb{R}^n)\cap L^2(\mathbb{R}^n)$, $\lambda>0$,  and $0<\alpha<2$, we define
$$
J_\lambda^\alpha(v):=\frac12\mathcal E_\alpha(v)-\int_{\mathbb{R}^n} (f-\lambda)v\,dx
$$
for $v\in X_\alpha$. Note that the second term is finite by H\"older's inequality and Lemma~\ref{lem:interp}. We further define the cone of nonnegative functions in $X_\alpha$,
$$
\mathcal K_\alpha:=\{v\in X_\alpha\colon v\ge0 \text{ almost everywhere on } \mathbb{R}^n\},
$$
and observe that, for $v \in \mathcal K_\alpha$, the energy can equivalently be written as
\[
J_\lambda^\alpha(v)=\tfrac12\mathcal E_\alpha(v)-\int_{\mathbb{R}^n} fv\,dx+\lambda\|v\|_{L^1(\mathbb{R}^n)}.
\]

Towards establishing the existence of minimizers, we first prove coercivity of $J_\lambda^\alpha$.
\begin{proposition}\label{prop:coerc}
If $f \in L^1(\mathbb{R}^n)\cap L^2(\mathbb{R}^n)$, $\lambda>0$,  and $0<\alpha<2$, then 
$$
J_\lambda^\alpha(v)\ge\frac14\mathcal E_\alpha(v)+\frac\lambda2\|v\|_{L^1(\mathbb{R}^n)}-c_0
$$
 for all $v\in\mathcal K_\alpha$, where
\[
c_0:=\frac{n^2}{(n+\alpha)^2}\left(\frac{2\alpha}{\lambda(n+\alpha)}\right)^{2\alpha/n}
\Big(K_{n,\alpha}\|f\|_{L^2}^2\Big)^{\frac{n+\alpha}{n}}<\infty .
\]
In particular, $\inf_{v \in \mathcal K_\alpha}J_{\lambda}^{\alpha}(v)\ge-c_0>-\infty$.
\end{proposition}

\begin{proof}
We use the following bound for $K,A,E \ge 0$, $s,t>0$, and $a,b,\gamma>0$ with $a+b+\gamma =1$: 
\begin{equation}\label{eq:young}
K A^{a}E^{b}\le sA+tE+\gamma\Big(K\big(\tfrac as\big)^{a}\big(\tfrac bt\big)^{b}\Big)^{1/\gamma},
\end{equation}
which follows from the weighted AM--GM inequality $x^{a}y^{b}z^{\gamma}\le ax+by+\gamma z$ applied to $x=sA/a$, $y=tE/b$, and $z=(K(a/s)^{a}(b/t)^{b})^{1/\gamma}$.

Let $v\in\mathcal K_\alpha$ and put $A:=\|v\|_{L^1(\mathbb{R}^n)}$ and $E:=\mathcal E_\alpha(v)$. By Cauchy-Schwarz and
Lemma~\ref{lem:interp}, we have
\[
\int_{\mathbb{R}^n} fv\,dx \le \|f\|_{L^2(\mathbb{R}^n)}\|v\|_{L^2(\mathbb{R}^n)}\le KA^{a}E^{b},
\]
where $K:=K_{n,\alpha}^{1/2}\|f\|_{L^2(\mathbb{R}^n)}$, $a:=\frac{\alpha}{n+\alpha}$, and $b:=\frac{n}{2(n+\alpha)}$. Note that $\gamma:=1-a-b=b>0$. 
Applying \eqref{eq:young} with $s=\lambda/2$ and $t=1/4$ to the above bound gives 
$$
\int_{\mathbb{R}^n}fv\,dx \leq \frac{\lambda}{2}A + \frac{1}{4}E + \gamma\Big(K\big(\tfrac as\big)^{a}\big(\tfrac bt\big)^{b}\Big)^{1/\gamma} = \frac{1}{4}\mathcal{E}_{\alpha}(v) +\frac{\lambda}{2}\|v\|_{L^1(\mathbb{R}^n)} + c_0.
$$ 
Hence 
$$
J_\lambda^\alpha(v)=\frac12\mathcal{E}_{\alpha}(v)-\int_{\mathbb{R}^n} fv\,dx +\lambda \|v\|_{L^1(\mathbb{R}^n)} \ge \frac14\mathcal{E}_{\alpha}(v)+\frac\lambda2\|v\|_{L^1(\mathbb{R}^n)}-c_0,
$$
as desired.
\end{proof}

We next show the existence of a unique minimizer of the energy $J_{\lambda}^{\alpha}$.
\begin{thm}
\label{thm:exist}
If $f \in L^1(\mathbb{R}^n)\cap L^2(\mathbb{R}^n)$, $\lambda>0$,  and $0<\alpha<2$, then there exists a unique $u\in\mathcal K_\alpha$ such that $J_\lambda^\alpha(u)=\min_{v \in \mathcal K_\alpha}J_{\lambda}^{\alpha}(v)$.
\end{thm}

\begin{proof}
We first prove the existence of such a minimizer.  Note the fact that $J_\lambda^\alpha(0)=0$ and Proposition~\ref{prop:coerc} implies that $m:=\inf_{v \in \mathcal{K}_{\alpha}} J_{\lambda}^{\alpha}(v) \in[-c_0,0]$. Let $\{v_k\}\subseteq \mathcal{K}_{\alpha}$ be a sequence such that $J_\lambda^\alpha(v_k)\to m$. By Proposition~\ref{prop:coerc}, we have that $\sup_k(\mathcal E_\alpha(v_k)+\|v_k\|_{L^1(\mathbb{R}^n)})<\infty$, and hence $\{v_k\}$ is bounded in $L^2(\mathbb{R}^n)$ by Lemma~\ref{lem:interp}. Passing to a subsequence if necessary, we may assume that $v_k\rightharpoonup u$ weakly to some $u \in L^2(\mathbb{R}^n)$. 

We will show that $u \in \mathcal{K}_{\alpha}$ and $J_{\lambda}^{\alpha}(u) \leq m$. Note that $u\ge0$, since weak convergence gives 
$$
	\int_{\mathbb{R}^n} u\varphi\,dx = \lim_{k\rightarrow \infty}\int_{\mathbb{R}^n} v_k\varphi\,dx\ge0
$$
for any nonnegative $\varphi\in C_c^\infty(\mathbb{R}^n)$.  Note also that  
$$
\int_{\mathbb{R}^n} u\varphi\,dx =\lim_{k\rightarrow \infty} \int_{\mathbb{R}^n} v_k\varphi\,dx \leq \liminf_{k\rightarrow \infty} \|v_k\|_{L^1(\mathbb{R}^n)}
$$ 
for any $\varphi\in C_c^\infty(\mathbb{R}^n)$ with $0\le\varphi\le1$, and so, taking a sequence of such $\varphi$ increasing to $1$, we have that $\|u\|_{L^1(\mathbb{R}^n)}\leq\liminf_{k\rightarrow\infty}\|v_k\|_{L^1(\mathbb{R}^n)}$ by monotone convergence. 
Moreover, $\widehat{v_k}\rightharpoonup \widehat u$ weakly in $L^2(\mathbb{R}^n)$ by Plancherel, and for each $R>0$ the map $g\mapsto(\int_{|\xi|<R}(2\pi|\xi|)^\alpha|g(\xi)|^2\,d\xi)^{1/2}$ is a continuous seminorm, hence weakly
lower semicontinuous, thus, letting $R \rightarrow\infty$ gives $\mathcal E_\alpha(u)\leq \liminf_{k\rightarrow \infty} \mathcal E_\alpha(v_k)$. This implies that $u \in \mathcal{K}_{\alpha}$. Further, since $f \in L^2(\mathbb{R}^n)$, we have that $\int_{\mathbb{R}^n} fv_k\,dx\to\int_{\mathbb{R}^n} fu\,dx$, and so $J_\lambda^\alpha(u)\le\liminf_{k\rightarrow\infty} J_\lambda^\alpha(v_k)=m$.

We finally show the uniqueness of the minimizer $u$.  The parallelogram identity gives
$$
J_\lambda^\alpha\Big(\frac{v+w}2\Big)=\frac12J_\lambda^\alpha(v)+\frac12J_\lambda^\alpha(w)-\frac18\mathcal E_\alpha(v-w).
$$
for any $v,w \in \mathcal{K}_{\alpha}$. If both $v$ and $w$ minimize $J_{\lambda}^{\alpha}$, then $\mathcal E_\alpha(v-w)\le0$, and so $v=w$.
\end{proof}

The following is a variational inequality satisfied by the unique minimizer of the energy.
\begin{thm}\label{thm:vi}
If $f \in L^1(\mathbb{R}^n)\cap L^2(\mathbb{R}^n)$, $\lambda>0$,  and $0<\alpha<2$, then the unique minimizer $u$ of $J_\lambda^\alpha$ satisfies
\begin{align}\label{eq:VI}
\mathcal E_\alpha(u,v-u)\ \ge\ \int_{\mathbb{R}^n}(f-\lambda)(v-u)\,dx
\end{align}
for all $v \in \mathcal{K}_{\alpha}$. Conversely, any $u\in\mathcal K_\alpha$ which satisfies \eqref{eq:VI} is the unique minimizer. Moreover, the unique minimizer $u$ satisfies 
$$
\mathcal E_\alpha(u)=\int_{\mathbb{R}^n}(f-\lambda)u\,dx.
$$
\end{thm}

\begin{proof}
The function $\phi(t):=J_\lambda^\alpha(u+t(v-u))$ is a quadratic polynomial with $\phi(t)\ge\phi(0)$ for $t \in [0,1]$, and so $\phi'(0)\ge0$. Since 
$$
\phi'(0) = \mathcal E_\alpha(u,v-u)\ - \ \int_{\mathbb{R}^n}(f-\lambda)(v-u)\,dx,
$$ 
this gives \eqref{eq:VI}. Conversely, assuming $u \in \mathcal{K}_{\alpha}$ satisfies \eqref{eq:VI}, we have 
$$
J_\lambda^\alpha(v)-J_\lambda^\alpha(u)=\frac12\mathcal E_\alpha(v-u)+\mathcal E_\alpha(u,v-u)-\int_{\mathbb{R}^n}(f-\lambda)(v-u)\,dx\ge0
$$
for all $v \in \mathcal{K}_{\alpha}$, and hence $u$ minimizes $J_{\lambda}^{\alpha}$ on $\mathcal{K}_{\alpha}$.  The energy identity for $u$ follows from applying \eqref{eq:VI} with $v=0$ and $v=2u$.
\end{proof}

\section{Lewy--Stampacchia type estimate}\label{sec:cap}

Define the distribution $\eta$ on $C_c^\infty(\mathbb{R}^n)$ by
$$
\langle\eta,\varphi\rangle:=\mathcal E_\alpha(u,\varphi)-\int_{\mathbb{R}^n}(f-\lambda)\varphi\,dx,
$$
i.e., $\eta=(-\Delta)^{\alpha/2} u-(f-\lambda)$ in the sense of distributions. Note $\eta$ is indeed a distribution as 
$$
|\langle\eta,\varphi\rangle|\le\mathcal E_\alpha(u)^{1/2}\mathcal E_\alpha(\varphi)^{1/2}
+\|f\|_{L^2(\mathbb{R}^n)}\|\varphi\|_{L^2(\mathbb{R}^n)}+\lambda\|\varphi\|_{L^1(\mathbb{R}^n)}
$$
for any $\varphi \in C_c^{\infty}(\mathbb{R}^n)$. Substituting $v=u+\varphi\in\mathcal K_\alpha$ in \eqref{eq:VI} gives
$
\langle\eta,\varphi\rangle\ge0
$
for all nonnegative $\varphi\in C_c^\infty(\mathbb{R}^n)$, and therefore $\eta$ is a locally-finite, positive Radon measure.

The following Lewy--Stampacchia type estimate is crucial to the establishment of the properties of $\mu$ and $u$ in the decomposition of Theorem~\ref{thm:decomp}.
\begin{thm}\label{thm:cap}
If $f \in L^1(\mathbb{R}^n)\cap L^2(\mathbb{R}^n)$, $\lambda>0$,  $0<\alpha<2$, and $\varphi\in C_c^\infty(\mathbb{R}^n)$ is nonnegative, then 
$$
 \langle\eta,\varphi\rangle \leq \int_{\mathbb{R}^n} (\lambda-f)^+\varphi\,dx \le \lambda\|\varphi\|_{L^1(\mathbb{R}^n)}.
$$
Hence, $\eta$ is absolutely continuous with respect to the Lebesgue measure with density satisfying
$0\le\eta\le(\lambda-f)^+\le\lambda$ almost everywhere. 
\end{thm}

\begin{proof}
For fixed $\varepsilon>0$, put $z:=u-\varepsilon\varphi\in X_\alpha$ and $v:=z^+$. Then $v\in\mathcal K_\alpha$ since $v\ge0$, $v\in L^1(\mathbb{R}^n)$ as $0\le v\le u$, and $\mathcal E_\alpha(v)\le\mathcal E_\alpha(z)<\infty$ by Lemma~\ref{lem:contr}. Since $v=z+z^-$, we have $v-u=-\varepsilon\varphi+z^-$, and \eqref{eq:VI} becomes
$$
\varepsilon\langle\eta,\varphi\rangle\ \le\ \mathcal E_\alpha(u,z^-) +\int_{\mathbb{R}^n} (\lambda-f)z^-\,dx .
$$
For the first term, the definition of $z$, Lemma~\ref{lem:negpart}, and Cauchy-Schwarz give
\begin{align*}
\mathcal E_\alpha(u,z^-) &= \mathcal E_\alpha(z,z^-) +\varepsilon\mathcal E_\alpha(\varphi,z^-) \\
&\leq -\mathcal{E}_{\alpha}(z^-) + \varepsilon\mathcal{E}_{\alpha}(\varphi,z^-) \\
&\leq -\mathcal{E}_{\alpha}(z^-) +\varepsilon\mathcal{E}_{\alpha}(\varphi)^{1/2}\mathcal{E}_{\alpha}(z^-)^{1/2}.
\end{align*}
Applying the bound $-b^2+\varepsilon ab\leq \varepsilon^2a^2/4$ with $a=\mathcal E_\alpha(\varphi)^{1/2}$ and $b=\mathcal E_\alpha(z^-)^{1/2}$, we have that 
$$
\mathcal E_\alpha(u,z^-)\leq \varepsilon^2\frac{\mathcal{E}_{\alpha}(\varphi)}{4}.
$$
For the second term, $u\ge0$ gives $0\le z^-=(\varepsilon\varphi-u)^+\le\varepsilon\varphi$, and so 
$$
\int_{\mathbb{R}^n}(\lambda-f)z^- \,dx \leq \int_{\mathbb{R}^n}(\lambda-f)^+z^-\,dx \leq \varepsilon\int_{\mathbb{R}^n}(\lambda-f)^+\varphi\,dx.
$$
Applying both of these bounds and dividing by $\varepsilon$ gives
$$
\langle \eta,\varphi\rangle \leq \varepsilon\frac{\mathcal{E}_{\alpha}(\varphi)}{4}+\int_{\mathbb{R}^n}(\lambda-f)^+\varphi\,dx,
$$
and the bound follows from letting $\varepsilon\rightarrow 0$. Finally, as $\eta$ is a nonnegative measure dominated by $(\lambda-f)^+dx$ on nonnegative test functions, it is absolutely continuous with density $\eta$ satisfying $0\leq \eta \leq (\lambda-f)^+$.
\end{proof}

Letting $\eta$ denote the density given by Theorem \ref{thm:cap}, define
\begin{align}\label{def:mu}
\mu:=\lambda-\eta
\end{align}
and observe that $\mu \in L^{\infty}(\mathbb{R}^n)$ with $0\le\mu\le\lambda$ almost everywhere, and 
\begin{equation}\label{eq:weakeq}
\int_{\mathbb{R}^n}\mu\varphi\,dx=\int_{\mathbb{R}^n} f\varphi\,dx-\mathcal E_\alpha(u,\varphi)
\end{equation}\
for all $\varphi\in C_c^\infty(\mathbb{R}^n)$. Note also that since $\eta\le(\lambda-f)^+$, we have $\mu\ge\min\{\lambda,f\}$.

The following lemma shows that $f$ and $\mu$ have the same mass.
\begin{lemma}\label{lem:mass}
If $f \in L^1(\mathbb{R}^n)\cap L^2(\mathbb{R}^n)$, $\lambda>0$,  $0<\alpha<2$, then $\mu$ as defined in \eqref{def:mu} is in $L^1(\mathbb{R}^n)$ with  
$$
\|\mu\|_{L^1(\mathbb{R}^n)} = \|f\|_{L^1(\mathbb{R}^n)}.
$$
\end{lemma}

\begin{proof}
Fix a radially nonincreasing function $\chi\in C_c^\infty(\mathbb{R}^n)$ such that $0\le\chi\le1$, $\chi\equiv1$ on $B(0,1)$, and $\text{supp}\,\chi\subseteq B(0,2)$. Set $\chi_\rho(x)=\chi(x/\rho)$ and note that $\chi_\rho$ increases to $1$ pointwise as $\rho\rightarrow \infty$. Using $\widehat{\chi_\rho}(\xi)=\rho^n\widehat\chi(\rho\xi)$, and substituting $\xi=\zeta/\rho$, we have
\[
\mathcal E_\alpha(u,\chi_\rho)=\int_{\mathbb{R}^n}(2\pi|\xi|)^\alpha\widehat u(\xi)\rho^n
\overline{\widehat\chi(\rho\xi)}\,d\xi =\rho^{-\alpha}\int_{\mathbb{R}^n}(2\pi|\zeta|)^\alpha\widehat u(\zeta/\rho)\overline{\widehat\chi(\zeta)}\,d\zeta ,
\]
hence $|\mathcal E_\alpha(u,\chi_\rho)|\le\rho^{-\alpha}\|u\|_{L^1(\mathbb{R}^n)}\int_{\mathbb{R}^n}(2\pi|\zeta|)^\alpha |\widehat\chi(\zeta)|\,d\zeta\to0$, where the integral is finite since $\widehat\chi$ is Schwartz. We now apply \eqref{eq:weakeq} with $\varphi=\chi_\rho$ and the fact that to see 
$$
	\int_{\mathbb{R}^n} \mu\chi_\rho\,dx = \int_{\mathbb{R}^n} f\chi_{\rho}\,dx + \mathcal E_\alpha(u,\chi_\rho)
$$
The above integrals converge to $\|\mu\|_{L^1(\mathbb{R}^n)}$ and $\|f\|_{L^1(\mathbb{R}^n)}$, respectively, by monotone convergence, and the final term vanishes as $\rho \rightarrow \infty$ by the above discussion. 
\end{proof}

The following lemma establishes the regularity of the function $u$ in the decomposition.

\begin{lemma}\label{lem:reg}
If $f \in L^1(\mathbb{R}^n)\cap L^2(\mathbb{R}^n)$, $\lambda>0$,  $0<\alpha<2$, $\mu$ is defined as in \eqref{def:mu}, and $u\in\mathcal{K}_{\alpha}$ is the minimizer of $J_{\lambda}^{\alpha}$,  then $f-\mu\in L^1(\mathbb{R}^n)\cap L^2(\mathbb{R}^n)$ and $(2\pi|\xi|)^\alpha\widehat u=\widehat{f-\mu}$ almost everywhere.  Consequently, $(-\Delta)^{\alpha/2} u=f-\mu$ almost everywhere, $u\in H^{\alpha}(\mathbb{R}^n)$, and
\begin{equation}\label{eq:star}
\mathcal E_\alpha(u,\psi)=\int_{\mathbb{R}^n}(f-\mu)\psi\,dx
\end{equation}
for all $\psi \in X_{\alpha}$. Moreover, if $\alpha\ge1$, then $u\in H^1(\mathbb{R}^n)$. 
\end{lemma}

\begin{proof}
Since $\mu\in L^1(\mathbb{R}^n)\cap L^\infty(\mathbb{R}^n)\subseteq L^2(\mathbb{R}^n)$, we have that $f-\mu\in L^1(\mathbb{R}^n)\cap L^2(\mathbb{R}^n)$, and thus $\widehat{f-\mu}$ is a bounded, continuous $L^2(\mathbb{R}^n)$ function. Both sides of \eqref{eq:weakeq} are continuous in $\varphi$ for the norm $\|\varphi\|_{L^1(\mathbb{R}^n)}+\|\varphi\|_{L^2(\mathbb{R}^n)}+\mathcal E_\alpha(\varphi)^{1/2}$, and $C_c^\infty(\mathbb{R}^n)$ is dense in the Schwartz class for it, so \eqref{eq:weakeq} holds for all Schwartz functions $\varphi$. Parseval then gives that 
$$
\int_{\mathbb{R}^n} G\overline{\widehat\varphi}\,d\xi=0
$$
for all such $\varphi$, where $G:=(2\pi|\cdot|)^\alpha\widehat u-\widehat{f-\mu}$ is a tempered distribution, and so $G=0$ almost everywhere. This implies that $(2\pi|\cdot|)^\alpha\widehat u\in L^2(\mathbb{R}^n)$ and hence that $(-\Delta)^{\alpha/2}u \in L^2(\mathbb{R}^n)$. Moreover, $u\in L^2(\mathbb{R}^n)$ by Lemma~\ref{lem:interp}, and so $u\in H^\alpha(\mathbb{R}^n)$. If $\alpha\ge1$ then $(1+|\xi|)\lesssim (1+|\xi|^\alpha)$, and so $u\in H^1(\mathbb{R}^n)$. The identity \eqref{eq:star} follows from Parseval's identity applied to $(f-\mu,\psi)\in L^2(\mathbb{R}^n)\times L^2(\mathbb{R}^n)$.
\end{proof}

The following lemma which asserts that $\mu\equiv \lambda$ on the positive set of $u$ is an important part of the decomposition.

\begin{lemma}\label{lem:compl}
If $f \in L^1(\mathbb{R}^n)\cap L^2(\mathbb{R}^n)$, $\lambda>0$,  $0<\alpha<2$, $\mu$ is defined as in \eqref{def:mu}, and $u\in\mathcal{K}_{\alpha}$ is the minimizer of $J_{\lambda}^{\alpha}$,  then $\eta u=0$ almost everywhere. Hence, $\mu=\lambda$ almost everywhere on $\Omega:=\{x \in \mathbb{R}^n\colon u(x)>0\}$.
\end{lemma}

\begin{proof}
Taking $\psi=u$ in Lemma~\ref{lem:reg} gives that 
$$
\mathcal E_\alpha(u)=\int_{\mathbb{R}^n}(f-\mu)u\,dx.
$$ 
On the other hand, the energy identity of Theorem~\ref{thm:vi} gives that 
$$
\mathcal E_\alpha(u)=\int_{\mathbb{R}^n}(f-\lambda)u\,dx.
$$
Noting that all integrals converge absolutely, subtracting gives $\int_{\mathbb{R}^n}\eta u \,dx =0$, which in turn implies $\eta u=0$ almost everywhere since $\eta$ and $u$ are nonnegative. This implies that if $u(x)>0$, then $\eta = 0$ up to a set of measure $0$, i.e., $\mu = \lambda$ for almost every $x\in \Omega$. 
\end{proof}

\section{The decomposition theorem}\label{decomposition}

We now prove our main decomposition theorem, Theorem~\ref{thm:decomp}. 
\begin{proof}[Proof of Theorem \ref{thm:decomp}]
Let $u \in\mathcal{K}_{\alpha}$ be the minimizer of $J_{\lambda}^{\alpha}$ given from Theorem~\ref{thm:exist} and $\mu$ be as defined in \eqref{def:mu}.  Claim (1) follows from Lemma \ref{lem:mass} and Theorem~\ref{thm:cap}, which implies that $\eta \leq \lambda$ and so $\mu \leq \lambda$. Claim (2) follows from Lemma~\ref{lem:compl}.  Claim (3) also follows from a combination of Lemma \ref{lem:mass} and Lemma~\ref{lem:compl}:
$$
\lambda|\Omega|=\int_\Omega\mu\,dx\le\|\mu\|_{L^1(\mathbb{R}^n)}=\|f\|_{L^1(\mathbb{R}^n)}.
$$

For $\alpha\geq 1$, an application of Lemma~\ref{lem:stamp} with $c=0$ applied
to $u$, which is in $H^1(\mathbb{R}^n)\subseteq W^{1,1}_{\mathrm{loc}}(\mathbb{R}^n)$ by Lemma~\ref{lem:reg}, yields the assertion $\nabla u=0$ in $\mathbb{R}^n\setminus\Omega$.
\end{proof}


\section*{Artificial Intelligence Statement} 

The proof strategy was developed by Large Language Models (LLMs), through a combination of ChatGPT (GPT-5.6 Sol), Codex CLI and web interface, mathematical reasoning agents, and Claude Opus 5.0, in dialogues with the authors. The first author utilized the second two authors’ paper \cite{SS}, and initialized a mathematical reasoning attempt that coordinated between a Sol agent, a Danus \cite{LGSWLJJCCZ2026} and a Rethlas \cite{JGJWSLCWW2026} automated mathematical reasoning agents (both deployed using OpenAI Sol agents), and a Polya’s ``How to Solve It’’ \cite{PolyaBook} style agent to search for an answer to the open question of Stein discussed in the paper. While the initial attempt failed to reach a complete solution and only produced partial results, further dialogue with ChatGPT Sol web interface yielded an attempted complete solution involving variational inequalities on the torus and a transference principle that was significantly more complicated than the proof presented.  The second and third authors discussed this idea to assess its feasibility and concluded that it should be possible to prove the result directly in Euclidean space.  The second author prompted Claude Opus 5.0 to try this, and the response was a longer document that provided the basis for the proof idea used in the present paper.  The second and third authors then checked and rewrote the proofs where we felt the presentation was unnatural, performed a literature review to ensure proper citation of the ideas utilized by LLMs as far as possible, and wrote an introduction according to our perspective of how the result can be placed in the context of relevant literature.

The authors have independently verified, validated, and rewritten all parts of the paper influenced by LLM-generated material and take full responsibility for the mathematical content of the paper.\\

This work was supported in part by OpenAI API credits provided by Clemson University and administered by Clemson University Research Computing and Data (RCD).


\end{document}